\documentclass[11pt,reqno]{amsart}
\usepackage{amsaddr}
\usepackage[pdftex]{graphicx}
\usepackage{pgfplots}
\usepackage{orcidlink}
\usepackage{geometry}
\usepackage{stix}
\usepackage{subcaption}
\usepackage{caption}
\usepackage{amsmath,amsthm}
\usepackage{cite}
\usepackage{hyperref}
\hypersetup{
    colorlinks=true, 
    linktoc=all,     
    allcolors=blue,  
}
\usepackage{url}
\newtheorem{theorem}{Theorem}
\newtheorem{lemma}{Lemma}
\newtheorem{proposition}{Proposition}
\newtheorem{corollary}{Corollary}

\theoremstyle{definition}
\newtheorem{definition}{Definition}
\newtheorem{example}{Example}

\pgfplotsset{compat=1.18}
\definecolor{wewdxt}{rgb}{0.45,0.45,0.45}
\definecolor{uuuuuu}{rgb}{0.3,0.3,0.3}
\definecolor{zzttqq}{rgb}{0.6,0.2,0.}
\definecolor{qqwwtt}{rgb}{0.,0.4,0.2}
\definecolor{xfqqff}{rgb}{0.5,0.,1.}
\definecolor{yqyqyq}{rgb}{0.5,0.5,0.5}
\definecolor{xdxdff}{rgb}{0.5,0.5,1.}
\definecolor{ududff}{rgb}{0.3,0.3,1.}
\definecolor{qqwuqq}{rgb}{0.,0.4,0.}

\begin{document}
\pagenumbering{arabic}

\title[An Affine Approach to Metric Invariants of Poncelet Polygons]{An Affine Approach to Metric Invariants of Poncelet Polygons}
\author{Mohammad Hassan Murad\orcidlink{0000-0002-8293-5242}}
\address{Department of Mathematics\\
The University of Texas at Arlington, Arlington, TX, USA}
\email{mohammad.murad2@uta.edu}

\begin{abstract}
We prove that the total area of the power circles associated with an odd $p$-gon inscribed in and circumscribed about a pair of homothetic ellipses remains invariant throughout the  Poncelet family. Our proof is based on a simple affine averaging principle, which also yields several related quadratic invariants, including explicit formulas for the sums of the squared distances from the center to the vertices, to the side midpoints, as well as for the sum of the squared side lengths. As an application, we show that a convex Poncelet pentagon and the corresponding star Poncelet pentagon, both circumscribed about the same inellipse, have equal total power-circle area. These results unify several metric invariants of odd Poncelet polygons within a common affine-geometric framework.
\end{abstract}



\maketitle{}

\section{Introduction}\label{sec:intro}
\noindent
We begin with the following definition.

\begin{definition}\label{def:powcirc}
Let $P=A_0A_1\cdots A_{2n}$ be a $(2n+1)$-gon. The circles passing through the vertex $A_k$ and centered at the midpoint $M_k$ of the segment
$\overline{A_{n+k}A_{n+k+1}}$, where the indices are taken modulo $2n+1$, are called the \emph{power circles} of $P$. The segment $\overline{A_kM_k}$ is called the \emph{median} corresponding to the vertex $A_k$, and its length is denoted by $m_k$. See Figure~\ref{fig:powercircle}. In the special case of a triangle, each power circle passes through a vertex and is centered at the midpoint of the opposite side.
\end{definition}

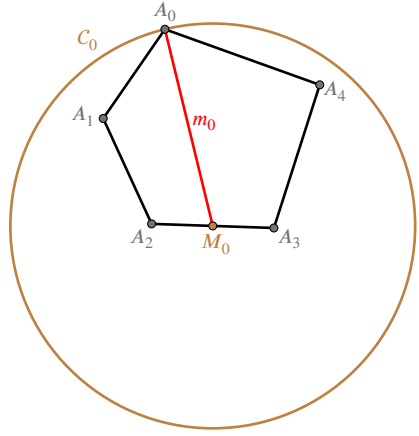
\begin{figure}
    \centering
\begin{tikzpicture}[scale=1.5]
\clip(-2,-3) rectangle (2,1.5);
\draw [line width=1.pt] (-0.4501785990726699,0.8929385359233682)-- (-0.994333036549735,0.10630998271932529);
\draw [line width=1.pt] (-0.994333036549735,0.10630998271932529)-- (-0.5675140507624962,-0.8233637119688618);
\draw [line width=1.pt] (-0.5675140507624962,-0.8233637119688618)-- (0.5110907898893505,-0.8595267328537721);
\draw [line width=1.pt] (0.5110907898893505,-0.8595267328537721)-- (0.9152866061016331,0.4028032133572842);
\draw [line width=1.pt] (0.9152866061016331,0.4028032133572842)-- (-0.4501785990726699,0.8929385359233682);
\draw [line width=1.pt,brown] (-0.02821163043657282,-0.8414452224113169) circle (1.7849770709437935cm);
\draw [line width=1.pt,red] (-0.4501785990726699,0.8929385359233682)-- (-0.02821163043657282,-0.8414452224113169);
\begin{scriptsize}
\draw [fill=wewdxt] (-0.4501785990726699,0.8929385359233682) circle (1.0pt);
\draw[color=wewdxt] (-0.45,1.05) node {$A_0$};
\draw [fill=wewdxt] (-0.994333036549735,0.10630998271932529) circle (1.0pt);
\draw[color=wewdxt] (-1.17,0.13) node {$A_1$};
\draw [fill=wewdxt] (-0.5675140507624962,-0.8233637119688618) circle (1.0pt);
\draw[color=wewdxt] (-0.65,-0.95) node {$A_2$};
\draw [fill=wewdxt] (0.5110907898893505,-0.8595267328537721) circle (1.0pt);
\draw[color=wewdxt] (0.65,-0.95) node {$A_3$};
\draw [fill=wewdxt] (0.9152866061016331,0.4028032133572842) circle (1.0pt);
\draw[color=wewdxt] (1.05,0.35) node {$A_4$};
\draw [fill=brown] (-0.02821163043657282,-0.8414452224113169) circle (1.0pt);
\draw[color=brown] (0,-1) node {$M_0$};
\draw[color=brown] (-1.13,0.8) node {$\mathcal C_0$};
\draw[color=red] (-0.1,0.07) node {$m_0$};
\end{scriptsize}
\end{tikzpicture}
    \caption{The power circle $\mathcal C_0$ corresponding to the vertex $A_0$ of a pentagon $A_0A_1A_2A_3A_4$ is shown in brown. Here $m_0$ denotes the length of the median $\overline{A_0M_0}$.}
    \label{fig:powercircle}
\end{figure}

In a recent paper \cite{Dragovic-Murad2026}, we established the following result concerning the total area of the power circles associated with a family of Poncelet triangles.

\begin{theorem}\label{thm:invareafoc}
Let $\mathcal P$ be a family of triangles inscribed in and circumscribed \, about a pair of homothetic ellipses. Then the sum of the areas of the power circles of a triangle in $\mathcal P$ remains invariant throughout the family.
\end{theorem}

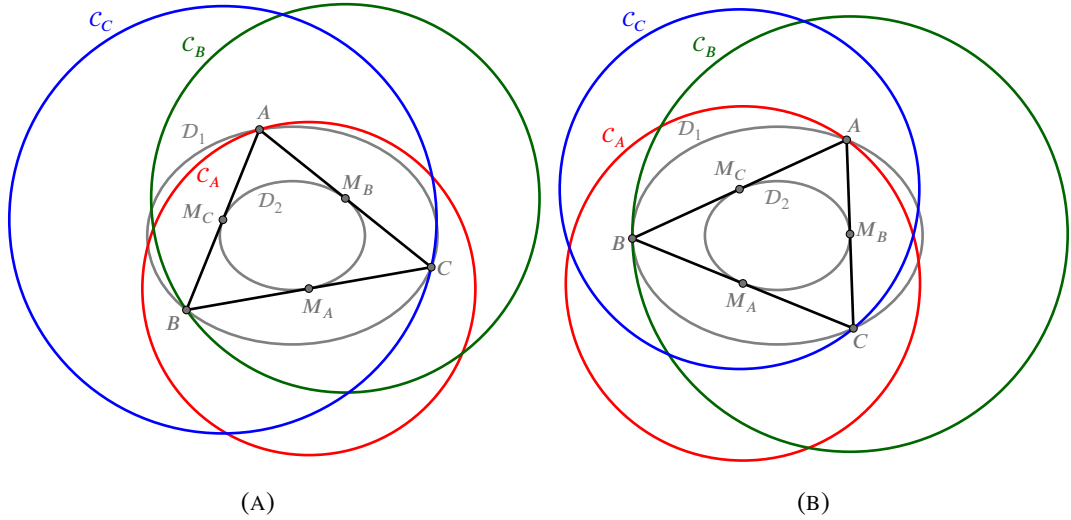
\begin{figure}
    \centering
    \begin{subfigure}[b]{0.48\textwidth}
    \centering
\begin{tikzpicture}[scale=0.48]
\clip(-8.5,-6.5) rectangle (7,7);
\draw [rotate around={0.:(0.,0.)},line width=1.pt,gray] (0.,0.) ellipse (4.cm and 3.cm);
\draw [rotate around={0.:(0.,0.)},line width=1.pt,gray] (0.,0.) ellipse (2.cm and 1.5cm);
\draw [line width=1.pt] (-0.9082402128286807,2.921642276210079)-- (-2.9195018028770203,-2.0507404608912068);
\draw [line width=1.pt] (-2.9195018028770203,-2.0507404608912068)-- (3.827742015705701,-0.8709018153188719);
\draw [line width=1.pt] (3.827742015705701,-0.8709018153188719)-- (-0.9082402128286807,2.921642276210079);
\draw [line width=1.pt,color=red] (0.4541201064143403,-1.4608211381050393) circle (4.589336685977449cm);
\draw [line width=1.pt,color=qqwuqq] (1.4597509014385102,1.0253702304456036) circle (5.351664342390273cm);
\draw [line width=1.pt,color=blue] (-1.9138710078528505,0.43545090765943617) circle (5.888351004239643cm);
\begin{scriptsize}
\draw[color=gray] (-2.7,2.8) node {$\mathcal D_1$};
\draw[color=gray] (-0.6,0.9) node {$\mathcal D_2$};
\draw [fill=wewdxt] (-0.9082402128286807,2.921642276210079) circle (3.0pt);
\draw[color=wewdxt] (-0.8,3.4) node {$A$};
\draw [fill=wewdxt] (-2.9195018028770203,-2.0507404608912068) circle (3.0pt);
\draw[color=wewdxt] (-3.3,-2.4) node {$B$};
\draw [fill=wewdxt] (3.827742015705701,-0.8709018153188719) circle (3.0pt);
\draw[color=wewdxt] (4.2,-0.9) node {$C$};
\draw [fill=wewdxt] (0.4541201064143403,-1.4608211381050393) circle (3.0pt);
\draw[color=wewdxt] (0.7,-2) node {$M_A$};
\draw [fill=wewdxt] (1.4597509014385102,1.0253702304456036) circle (3.0pt);
\draw[color=wewdxt] (1.8,1.4) node {$M_B$};
\draw [fill=wewdxt] (-1.9138710078528505,0.43545090765943617) circle (3.0pt);
\draw[color=wewdxt] (-2.6,0.6) node {$M_C$};
\draw[color=red] (-2.3,1.6) node {$\mathcal C_A$};
\draw[color=qqwuqq] (-2.7,5.2) node {$\mathcal C_B$};
\draw[color=blue] (-5.25,5.9) node {$\mathcal C_C$};
\end{scriptsize}
\end{tikzpicture}
    \caption{}
    \label{fig:trianglehomell(A)}
\end{subfigure}
\begin{subfigure}[b]{0.48\textwidth}
    \centering
    \begin{tikzpicture}[scale=0.48]
\clip(-6.5,-6.5) rectangle (8.5,6.5);
\draw [rotate around={0.:(0.,0.)},line width=1.pt,gray] (0.,0.) ellipse (4.cm and 3.cm);
\draw [rotate around={0.:(0.,0.)},line width=1.pt,gray] (0.,0.) ellipse (2.cm and 1.5cm);
\draw [line width=1.pt] (1.905946159523796,2.6375452973950093)-- (-3.9985480546635777,-0.08082430440275834);
\draw [line width=1.pt] (-3.9985480546635777,-0.08082430440275834)-- (2.0926018951397825,-2.556720992992251);
\draw [line width=1.pt] (2.0926018951397825,-2.556720992992251)-- (1.905946159523796,2.6375452973950093);
\draw [line width=1.pt,color=red] (-0.9529730797618976,-1.3187726486975047) circle (4.881175156387197cm);
\draw [line width=1.pt,color=qqwuqq] (1.9992740273317893,0.04041215220137917) circle (5.99904725816368cm);
\draw [line width=1.pt,color=blue] (-1.0463009475698908,1.2783604964961255) circle (4.9558612860921825cm);
\begin{scriptsize}
\draw[color=gray] (-2.4,2.9) node {$\mathcal D_1$};
\draw[color=gray] (0,1) node {$\mathcal D_2$};
\draw [fill=wewdxt] (1.905946159523796,2.6375452973950093) circle (3.0pt);
\draw[color=wewdxt] (2.1,3.06) node {$A$};
\draw [fill=wewdxt] (-3.9985480546635777,-0.08082430440275834) circle (3.0pt);
\draw[color=wewdxt] (-4.4,-0.2) node {$B$};
\draw [fill=wewdxt] (2.0926018951397825,-2.556720992992251) circle (3.0pt);
\draw[color=wewdxt] (2.3,-2.9) node {$C$};
\draw [fill=wewdxt] (-0.9529730797618976,-1.3187726486975047) circle (3.0pt);
\draw[color=wewdxt] (-1,-1.85) node {$M_A$};
\draw [fill=wewdxt] (1.9992740273317893,0.04041215220137917) circle (3.0pt);
\draw[color=wewdxt] (2.6,0.1) node {$M_B$};
\draw [fill=wewdxt] (-1.0463009475698908,1.2783604964961255) circle (3.0pt);
\draw[color=wewdxt] (-1.3,1.8) node {$M_C$};
\draw[color=red] (-4.5,2.7) node {$\mathcal C_A$};
\draw[color=qqwuqq] (-2,5.2) node {$\mathcal C_B$};
\draw[color=blue] (-3.93,5.9) node {$\mathcal C_C$};
\end{scriptsize}
\end{tikzpicture}
     \caption{}
    \label{fig:trianglehomell(B)}
    \end{subfigure}
        \caption{In each figure, the triangle $\triangle ABC$ is inscribed in and circumscribed about a pair of homothetic ellipses $\mathcal{D}_1$ and $\mathcal{D}_2$, respectively. The sum of the areas of the power circles of the triangle $\triangle ABC$ in both figures are equal.}
    \label{fig:trianglehomell}
\end{figure}
The existence of such families is guaranteed by Poncelet's Porism, one of the classical results of projective geometry. It asserts that if one polygon can be inscribed in one conic and circumscribed about another, then infinitely many such polygons exist, obtained by continuously moving the initial vertex along the outer conic. Throughout this paper, we refer to such a family as a Poncelet family.

The principal goal of the present paper is to extend this phenomenon to odd Poncelet polygons.

\begin{theorem}\label{thm:homellipse}
Let $p\ge3$ be an odd integer and let $\mathcal P$ denote a family of $p$-gons inscribed in and circumscribed about a pair of homothetic ellipses. Then the sum of the areas of the power circles of a $p$-gon in $\mathcal P$ remains invariant throughout the family.
\end{theorem}

Special cases of this result were first observed by Dan Reznik in a series of computer-assisted experiments \cite{Reznik2021,Reznik2024}. Theorem~\ref{thm:homellipse} extends these observations from triangles to arbitrary odd polygons and applies equally to simple and crossed Poncelet polygons.

The key ingredient is an affine averaging principle for rotationally symmetric vectors. This \,principle not only provides a concise proof of Theorem~\ref{thm:homellipse}, but also yields explicit formulas for several other quadratic invariants of odd Poncelet polygons, including the sums of the squared \,distances from the vertices and the side midpoints to the center, and the sum of the squared side lengths. These invariants arise naturally from the same affine-geometric argument, placing the power-circle invariant within a unified framework.

As an application, we prove that a convex Poncelet pentagon and the corresponding Poncelet star pentagon circumscribed about the same inellipse have equal total power-circle area.

Poncelet polygons and the geometric invariants associated with them have attracted considerable attention in recent years; see, for example,
\cite{Dragovic-Murad2026,Garciaetal.2026}
and the references therein.

\section{Affine transformations and homothetic ellipses}\label{sec:4}
\noindent
For a family $\mathcal P$ of odd $p$-gons, the problem of determining whether the total area of the power circles remains invariant may be reduced to a purely metric question. Indeed, if $m_k$ denotes the length of the median corresponding to the vertex $A_k$, then the radius of the $k$th power circle is $m_k$. Consequently,
\[
\sum_{k=0}^{p-1}\operatorname{Area}(\mathcal C_k)
=
\pi\sum_{k=0}^{p-1}m_k^2,
\]
where $\mathcal C_k$ is the power circle corresponding to the vertex $A_k$. Thus, it suffices to prove that
\[
\sum_{k=0}^{p-1}m_k^2
\]
is invariant throughout the family $\mathcal P$.

The key observation is that every odd $p$-gon inscribed in and circumscribed about a pair of \,homothetic ellipses is the image of a regular polygon under an affine transformation. This allows the problem to be reduced to the corresponding regular polygon, where the rotational symmetry implies that the horizontal and vertical contributions to certain quadratic sums are equal. An \,\,elementary affine averaging argument then transfers these identities to the original Poncelet family, yielding not only the desired power-circle invariant but also several related metric invariants.

We recall some basic facts concerning regular polygons. Throughout, we use the \emph{Schläfli symbol} to denote both convex and star regular polygons.

\subsection{Regular polygons}\label{sec.7.2}
\begin{definition}\label{def:regpol}
A \textit{regular polygon} is denoted by the Schläfli symbol $\{p/q\}$, where $p$ and $q$ are coprime positive integers. Its vertices are obtained by successive rotations of a vertex about the center $O$ through the angle
\[
\theta=\frac{2q\pi}{p}.
\]
The integer $q$ is called the \textit{density} of the polygon.

If $q=1$, the polygon is a regular convex $p$-gon. If
\[
1<q\le \left\lfloor \frac{p}{2}\right\rfloor,
\]
the polygon is a non-convex regular polygon, called a \textit{regular star polygon}. See Coxeter \cite{Coxeter1991}.
\end{definition}

\begin{theorem}\label{thm:reqpolpq}
The circumradius $R$ and the inradius $r$ of a regular polygon $\{p/q\}$ satisfy 
\begin{equation}\label{eq:reqpolpq}
  r=R\cos \left(\frac{q \pi}{p}\right).
\end{equation}
\end{theorem}
See Figure \ref{fig:regpolypq}.
\begin{figure}[htbp]
  \centering
\begin{subfigure}[b]{0.45\textwidth}
    \centering
\begin{tikzpicture}[scale=1]
\clip(-3.6,-3.6) rectangle (3.6,3.6);
\draw [line width=1.pt,gray] (0.,0.) circle (1.cm);
\draw [line width=1.pt,gray] (0.,0.) circle (3.23606797749979cm);
\draw [line width=1.pt,gray] (0.,0.) circle (1.2360679774997898cm);
\draw [line width=1.pt] (0.,3.23606797749979)-- (1.902113032590307,-2.6180339887498953);
\draw [line width=1.pt] (1.902113032590307,-2.6180339887498953)-- (-3.0776835371752527,1.);
\draw [line width=1.pt] (-3.0776835371752527,1.)-- (3.0776835371752536,1.);
\draw [line width=1.pt] (3.0776835371752536,1.)-- (-1.9021130325903075,-2.618033988749895);
\draw [line width=1.pt] (-1.9021130325903075,-2.618033988749895)-- (0.,3.2360679774997894);
\begin{scriptsize}
\draw [fill=blue] (0.726542528005361,1.) circle (1.5pt);
\draw[color=blue] (0.86,1.22) node {$A_3$};
\draw [fill=blue] (0.,-1.2360679774997898) circle (1.5pt);
\draw[color=blue] (0,-1.5) node {$A_0$};
\draw [fill=blue] (1.1755705045849463,-0.38196601125010543) circle (1.5pt);
\draw[color=blue] (1.4,-0.45) node {$A_4$};
\draw [fill=blue] (-0.726542528005361,1.) circle (1.5pt);
\draw[color=blue] (-0.9,1.22) node {$A_2$};
\draw [fill=blue] (-1.1755705045849465,-0.38196601125010476) circle (1.5pt);
\draw[color=blue] (-1.42,-0.41) node {$A_1$};
\draw [fill=red] (0.,3.23606797749979) circle (1.5pt);
\draw[color=red] (0,3.46) node {$B_0$};
\draw [fill=red] (-1.9021130325903075,-2.618033988749895) circle (1.5pt);
\draw[color=red] (-2.05,-2.9) node {$B_1$};
\draw [fill=red] (3.0776835371752536,1.) circle (1.5pt);
\draw[color=red] (3.3,1.05) node {$B_2$};
\draw [fill=red] (-3.0776835371752527,1.) circle (1.5pt);
\draw[color=red] (-3.35,1.05) node {$B_3$};
\draw [fill=red] (1.902113032590307,-2.6180339887498953) circle (1.5pt);
\draw[color=red] (2.05,-2.8) node {$B_4$};
\end{scriptsize}
\end{tikzpicture}
    \caption{$\{5/1\}$, $\{5/2\}$- a pentagram.}
    \label{fig:regpolypq(A)}
  \end{subfigure}
\begin{subfigure}[b]{0.45\textwidth}
\begin{tikzpicture}[scale=0.7]
\clip(-5.5,-5.5) rectangle (5.5,5.5);
\draw [line width=1.pt,gray] (0.,0.) circle (1.1099162641747424cm);
\draw [line width=1.pt,gray] (0.,0.) circle (1.cm);
\draw [line width=1.pt,gray] (0.,0.) circle (1.1099162641747424cm);
\draw [line width=1.pt,gray] (0.,0.) circle (1.6038754716096764cm);
\draw [line width=1.pt,gray] (0.,0.) circle (4.4939592074349335cm);
\draw [line width=1.pt] (0.01127128059024319,4.493945072686477)-- (-1.9600047644393377,-4.044014179186853);
\draw [line width=1.pt] (-1.9600047644393377,-4.044014179186853)-- (3.520535266810272,2.7931166809201406);
\draw [line width=1.pt] (3.520535266810272,2.7931166809201406)-- (-4.383780583057846,-0.9890081686691106);
\draw [line width=1.pt] (-4.383780583057846,-0.9890081686691106)-- (4.378764391290196,-1.0109855407760346);
\draw [line width=1.pt] (4.378764391290196,-1.0109855407760346)-- (-3.5064802098064614,2.810741164946508);
\draw [line width=1.pt] (-3.5064802098064614,2.810741164946508)-- (1.939694618612934,-4.053795029921128);
\draw [line width=1.pt] (1.939694618612934,-4.053795029921128)-- (0.01127128059024319,4.493945072686477);
\begin{scriptsize}
\draw [fill=brown] (0.0010197255316076338,1.1099995316034328) circle (2.0pt);
\draw[color=brown] (0.05,0.7) node {$A_0$};
\draw [fill=brown] (-0.8677543147901199,0.692172268415262) circle (2.0pt);
\draw[color=brown] (-0.55,0.55) node {$A_1$};
\draw [fill=brown] (-1.0821206509908883,-0.24721427284656555) circle (2.0pt);
\draw[color=brown] (-0.7,-0.1) node {$A_2$};
\draw [fill=brown] (-0.482931833005516,-0.9994382645616148) circle (2.0pt);
\draw[color=brown] (-0.3,-0.65) node {$A_3$};
\draw [fill=brown] (0.48109470743306904,-1.0003238887879713) circle (2.0pt);
\draw[color=brown] (0.35,-0.6) node {$A_4$};
\draw [fill=brown] (1.0821198455546628,-0.247217798422267) circle (2.0pt);
\draw[color=brown] (0.7,-0.1) node {$A_5$};
\draw [fill=brown] (0.869024607637719,0.6905767381834611) circle (2.0pt);
\draw[color=brown] (0.6,0.5) node {$A_6$};
\draw [fill=blue] (0.0025177425392164388,-1.5999980190526817) circle (2.0pt);
\draw[color=blue] (-0.05,-2.1) node {$B_0$};
\draw [fill=blue] (1.5606968104485155,0.3524563318424436) circle (2.0pt);
\draw[color=blue] (2,0.45) node {$B_6$};
\draw [fill=blue] (-0.6970613037207168,1.4401755236273022) circle (2.0pt);
\draw[color=blue] (-0.9,1.9) node {$B_5$};
\draw [fill=blue] (-1.2528136620602912,-0.995217528058606) circle (2.0pt);
\draw[color=blue] (-1.6,-1.25) node {$B_4$};
\draw [fill=blue] (1.2532997367092586,-0.994605333770387) circle (2.0pt);
\draw[color=blue] (1.6,-1.3) node {$B_3$};
\draw [fill=blue] (0.6925253682554742,1.442362164757042) circle (2.0pt);
\draw[color=blue] (0.9,1.9) node {$B_2$};
\draw [fill=blue] (-1.558919429286089,0.360236329364952) circle (2.0pt);
\draw[color=blue] (-1.95,0.5) node {$B_1$};
\draw [fill=red] (0.01127128059024319,4.493945072686477) circle (2.0pt);
\draw[color=red] (0.03,4.8) node {$C_0$};
\draw [fill=red] (-1.9600047644393377,-4.044014179186853) circle (2.0pt);
\draw[color=red] (-2.11,-4.3) node {$C_1$};
\draw [fill=red] (3.520535266810272,2.7931166809201406) circle (2.0pt);
\draw[color=red] (3.8,3.05) node {$C_2$};
\draw [fill=red] (-4.383780583057846,-0.9890081686691106) circle (2.0pt);
\draw[color=red] (-4.74,-1.1) node {$C_3$};
\draw [fill=red] (4.378764391290196,-1.0109855407760346) circle (2.0pt);
\draw[color=red] (4.7,-1.1) node {$C_4$};
\draw [fill=red] (-3.5064802098064614,2.810741164946508) circle (2.0pt);
\draw[color=red] (-3.8,3.1) node {$C_5$};
\draw [fill=red] (1.939694618612934,-4.053795029921128) circle (2.0pt);
\draw[color=red] (2.15,-4.3) node {$C_6$};
\end{scriptsize}
\end{tikzpicture}
    \caption{$\{7/1\}$, $\{7/2\}$, $\{7/3\}$.}
    \label{fig:regpolypq(B)}
\end{subfigure}
  \caption{Some regular polygons.}
  \label{fig:regpolypq}
\end{figure}
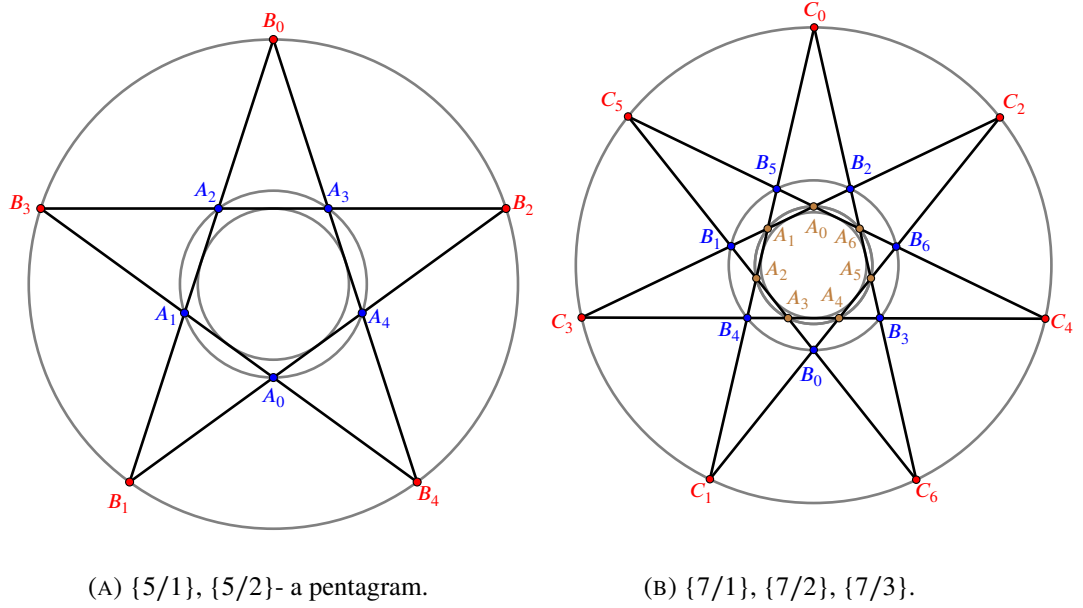

For the purposes of this paper, we extend the notion of density from regular polygons. We say that a polygon has \emph{density} $q$ if its vertices
are connected according to the same cyclic ordering as the regular polygon $\{p/q\}$.

Let $p\geq 3$ be an odd integer. For a $p$-gon of density $q$,  we define
\[
K_{p,q}:=\sum_{k=0}^{p-1}m_k^2
\]
where $m_k$ was defined in Definition \ref{def:powcirc}. 

Throughout this section, $\{p/q\}$ denotes the regular polygon of density $q$ and unit inradius.

\begin{proposition}\label{prop:regularmetric}
Let $P=\{p/q\}$ be a regular odd $p$-gon of unit inradius. Then
\begin{align}
|OA_k|&=\sec\!\left(\frac{q\pi}{p}\right),\label{eq:OA}\\
|OM_k|&=1,\label{eq:OM}\\
|A_kA_{k+1}|&=
2\tan\left(\frac{q\pi}{p}\right),\label{eq:side}\\
|A_kM_k|
&=
\sec\left(\frac{q\pi}{p}\right)-(-1)^q.\label{eq:median}
\end{align}
\end{proposition}

\begin{proof}
Let
\[
\varphi=\frac{q\pi}{p}.
\]
Since the inradius equals one, the circumradius of $P$ is
\[
R=\sec\varphi.
\]
Hence
\[
|OA_k|=R=\sec\varphi,
\]
while every midpoint $M_k$ lies on the incircle, giving
\[
|OM_k|=1.
\]
Furthermore,
\[
|A_kA_{k+1}|
=
2R\sin\varphi
=
2\tan\varphi.
\]

It remains to compute the median length. Identifying the plane with $\mathbb C$, the complex coordinate of the vertex $A_k$ can be presented by $z_k$ where
\[
z_k=R e^{2ik\varphi},
\qquad
k=0,1,\dots,2n.
\]

Since $M_0$ is the midpoint of $\overline{A_nA_{n+1}}$, its complex coordinate is
\[
w_0
=
\frac{1}{2}(z_n+z_{n+1})
=
\frac{R}{2}
\left(
e^{2in\varphi}
+
e^{2i(n+1)\varphi}
\right).
\]
Applying the elementary identity
\[
\frac{e^{i\theta_1}+e^{i\theta_2}}{2}
=
\cos\!\left(\frac{\theta_1-\theta_2}{2}\right)
e^{i(\theta_1+\theta_2)/2},
\]
we obtain
\[
w_0=R\cos\varphi\,e^{i(2n+1)\varphi}.
\]
Since $(2n+1)\varphi=q\pi$ and $R=\sec \varphi$, it follows that
\[
w_0
=
(-1)^q.
\]
Therefore
\[
|A_0M_0|
=
|R-(-1)^q|
=
\sec\varphi-(-1)^q,
\]
completing the proof.
\end{proof}

\subsection{An Affine Averaging Principle}\label{sec:affprin}

\begin{lemma}[Rotational symmetry]\label{lemm:rotsym}
Let
\[
\mathbf v_0,\mathbf v_1,\ldots,\mathbf v_{p-1}\in\mathbb R^2
\]
satisfy
\[
\mathbf v_{k+1}=R_\theta\mathbf v_k,
\qquad
\theta=\frac{2q\pi}{p},
\]
where $R_\theta$ denotes the matrix of rotation through the angle $\theta$,
$p$ is odd, and $\gcd(p,q)=1$. Writing
\[
\mathbf v_k=
\begin{pmatrix}
x_k\\
y_k
\end{pmatrix},
\]
one has
\[
\sum_{k=0}^{p-1}x_k^2
=
\sum_{k=0}^{p-1}y_k^2
=
\frac12
\sum_{k=0}^{p-1}\|\mathbf v_k\|^2.
\]
Moreover,
\[
\sum_{k=0}^{p-1}x_ky_k=0.
\]
\end{lemma}

\begin{proof}
Consider the symmetric matrix
\[
S=
\sum_{k=0}^{p-1}
\mathbf v_k\mathbf v_k^{T}.
\]
Since
\[
\mathbf v_{k+1}=R_\theta\mathbf v_k,
\]
we have
\[
R_\theta SR_\theta^T
=
\sum_{k=0}^{p-1}
R_\theta\mathbf v_k(R_\theta\mathbf v_k)^T
=
\sum_{k=0}^{p-1}
\mathbf v_{k+1}\mathbf v_{k+1}^T
=
S.
\]
Hence $S$ commutes with the nontrivial rotation $R_\theta$. Since
$\theta\neq0,\pi$, the only symmetric matrices commuting with
$R_\theta$ are scalar multiples of the identity. Thus
\[
S=\lambda I
\]
for some $\lambda\in\mathbb R$.

On the other hand,
\[
S
=
\begin{pmatrix}
\displaystyle\sum x_k^2 &
\displaystyle\sum x_ky_k\\[2ex]
\displaystyle\sum x_ky_k &
\displaystyle\sum y_k^2
\end{pmatrix}.
\]
Comparing entries gives
\[
\sum x_k^2=\sum y_k^2=\lambda,
\qquad
\sum x_ky_k=0.
\]
Since
\[
\sum\|\mathbf v_k\|^2
=
\sum x_k^2+\sum y_k^2,
\]
the first identity follows immediately.
\end{proof}

The following simple lemma is the key ingredient behind all of the metric invariants established in this paper.

\begin{lemma}[Affine averaging principle]\label{lemm:affineavgprin}
Let
\[
Q=
\begin{pmatrix}
\alpha^{-1} & 0\\
0 & \beta^{-1}
\end{pmatrix},
\]
and let $\mathbf v_k=
\begin{pmatrix}
x_k\\
y_k
\end{pmatrix}\in\mathbb R^2$,
$k=0,1,\dots,p-1$, be a family of rotationally symmetric vectors. Then
\begin{equation}\label{eq:affinesum}
\sum_{k=0}^{p-1}\|Q^{-1}\mathbf v_k\|^2
=
\frac{\alpha^2+\beta^2}{2}
\sum_{k=0}^{p-1}\|\mathbf v_k\|^2.
\end{equation}
\end{lemma}

\begin{proof}
By Lemma~\ref{lemm:rotsym},
\[
\sum_{k=0}^{p-1}x_k^2
=
\sum_{k=0}^{p-1}y_k^2
=
\frac12
\sum_{k=0}^{p-1}\|\mathbf v_k\|^2.
\]
Moreover,
\[
Q^{-1}\mathbf v_k=\begin{pmatrix}
    \alpha x_k\\
    \beta y_k
\end{pmatrix}
\]
gives
\[
\|Q^{-1}\mathbf v_k\|^2
=
\alpha^2x_k^2+\beta^2y_k^2.
\]
Summing over $k$ gives
\[
\sum_{k=0}^{p-1}\|Q^{-1}\mathbf v_k\|^2
=
\alpha^2\sum_{k=0}^{p-1}x_k^2
+
\beta^2\sum_{k=0}^{p-1}y_k^2
=
\frac{\alpha^2+\beta^2}{2}
\sum_{k=0}^{p-1}\|\mathbf v_k\|^2,
\]
which proves \eqref{eq:affinesum}.
\end{proof}

The hypothesis of Lemma~\ref{lemm:affineavgprin} is satisfied by every rotationally symmetric family of vectors associated with a regular polygon, such as the position vectors of the vertices, the side midpoints, the side vectors, and the median vectors. Therefore, all of the quadratic invariants established below follow from the same affine averaging principle.

\subsection{Homothetic Ellipses}

\begin{corollary}\label{cor:vertexsum}
Let $p \geq 3$ be an odd integer and let
\[
P=A_0A_1\cdots A_{p-1}
\]
be a $p$-gon of density $q$ inscribed in and circumscribed about a pair of homothetic ellipses whose inellipse has semi-axes lengths $\alpha$ and $\beta$. Let $O$ denote the common center of the ellipses. 
Then
\[
\sum_{k=0}^{p-1}|OA_k|^2
=
\frac{p}{2}\,(\alpha^2+\beta^2)
\sec^2\!\left(\frac{q\pi}{p}\right).
\]
\end{corollary}

\begin{proof}
Apply Lemma~\ref{lemm:affineavgprin} to the position vectors
$\overrightarrow{OA_k}$ and use Proposition~\ref{prop:regularmetric}.
\end{proof}

\begin{corollary}\label{cor:midpointsum}
Let $P$ be as in Corollary~\ref{cor:vertexsum}. Then
\[
\sum_{k=0}^{p-1}|OM_k|^2
=
\frac{p}{2}(\alpha^2+\beta^2),
\]
where $M_k$ denotes the midpoint of the side
$\overline{A_{n+k}A_{n+k+1}}$.
\end{corollary}

\begin{proof}
Apply Lemma~\ref{lemm:affineavgprin} to the position vectors $\overrightarrow{OM_k}$ and use Proposition~\ref{prop:regularmetric}.
\end{proof}

\begin{corollary}\label{cor:sidesum}
Let $P$ be as in Corollary~\ref{cor:vertexsum}. Then
\[
\sum_{k=0}^{p-1}|A_kA_{k+1}|^2
=
2p(\alpha^2+\beta^2)
\tan^2\!\left(\frac{q\pi}{p}\right).
\]
\end{corollary}

\begin{proof}
Apply Lemma~\ref{lemm:affineavgprin} to the side vectors
$\overrightarrow{A_kA_{k+1}}$ and use Proposition~\ref{prop:regularmetric}.
\end{proof}

\begin{corollary}\label{cor:mediansum}
Let $P$ be as in Corollary~\ref{cor:vertexsum}. Then
\[
\sum_{k=0}^{p-1}m_k^2
=
\frac{p}{2}(\alpha^2+\beta^2)
\left(
\sec\!\left(\frac{q\pi}{p}\right)-(-1)^q
\right)^2,
\]
where $m_k=|A_kM_k|$ denotes the median corresponding to the vertex $A_k$.
\end{corollary}

\begin{proof}
Apply Lemma~\ref{lemm:affineavgprin} to the median vectors
$\overrightarrow{A_kM_k}$ and use Proposition~\ref{prop:regularmetric}.
\end{proof}

\begin{proof}[Proof of Theorem~\ref{thm:homellipse}]
Since the radius of the power circle $\mathcal C_k$ equals $m_k$,
\[
\sum_{k=0}^{p-1}\operatorname{Area}(\mathcal C_k)
=
\pi\sum_{k=0}^{p-1}m_k^2.
\]
The result therefore follows directly from Corollary~\ref{cor:mediansum}.
\end{proof}

\medskip

\begin{example}\label{ex:equalms}
Let $m_{p,q}$ denote the length of the median of the regular $p$-gon of density $q$. Using the values
\[
\sec\!\left(\frac{\pi}{5}\right)
=\frac{4}{\sqrt5+1},
\qquad
\sec\!\left(\frac{2\pi}{5}\right)
=\sqrt5+1,
\]
Proposition~\ref{prop:regularmetric} yields
\[
m_{5,1}
=
\sec\!\left(\frac{\pi}{5}\right)+1
=
\sqrt5,
\]
and
\[
m_{5,2}
=
\sec\!\left(\frac{2\pi}{5}\right)-1
=
\sqrt5.
\]
Hence
\[
m_{5,1}=m_{5,2}=\sqrt5.
\]
Thus, the regular pentagon and the regular pentagram have the same median length.
\end{example}

This remarkable coincidence underlies the equality of the total power-circle areas for the convex and star Poncelet pentagons established in the following result.

\begin{figure}
    \centering
\begin{tikzpicture}[scale=0.78]
\clip(-6,-5) rectangle (6.,5);
\draw [rotate around={0.:(0.,0.)},line width=1.pt,gray] (0.,0.) ellipse (1.618033988749895cm and 1.2135254915624212cm);
\draw [rotate around={0.:(0.,0.)},line width=1.pt,gray] (0.,0.) ellipse (2.cm and 1.5cm);
\draw [rotate around={0.:(0.,0.)},line width=1.pt,gray] (0.,0.) ellipse (5.236067977499788cm and 3.9270509831248424cm);
\draw [line width=1.pt] (4.232075097435111,2.3123788361773547)-- (-5.236063578787058,-0.005090279250499996);
\draw [line width=1.pt] (-5.236063578787058,-0.005090279250499996)-- (4.240053740297761,-2.304142591337816);
\draw [line width=1.pt] (4.240053740297761,-2.304142591337816)-- (-1.6244874871408412,3.7332713069613477);
\draw [line width=1.pt] (-1.6244874871408412,3.7332713069613477)-- (-1.6115777718049742,-3.7364172725503857);
\draw [line width=1.pt] (-1.6115777718049742,-3.7364172725503857)-- (4.232075097435111,2.3123788361773547);
\begin{scriptsize}
\draw[color=gray] (0.5,0.9) node {$\mathcal{D}$};
\draw[color=gray] (-0.5,1.7) node {$\mathcal{D}_1$};
\draw[color=gray] (-3.1,3.5) node {$\mathcal{D}_2$};
\draw [fill=blue] (1.9999983198412434,0.0019443136614628858) circle (1.5pt);
\draw[color=blue] (2.3,0) node {$A_0$};
\draw [fill=blue] (0.6155679333156783,1.4271844019620674) circle (1.5pt);
\draw[color=blue] (0.67,1.7) node {$A_1$};
\draw [fill=blue] (-1.6195564146676258,0.8801041549647871) circle (1.5pt);
\draw[color=blue] (-1.87,1.14) node {$A_2$};
\draw [fill=blue] (-1.6165088442781896,-0.8832501205538245) circle (1.5pt);
\draw[color=blue] (-1.85,-1.1) node {$A_3$};
\draw [fill=blue] (0.6204990057888937,-1.4259827500344926) circle (1.5pt);
\draw[color=blue] (0.75,-1.7) node {$A_4$};
\draw [fill=red] (4.232075097435111,2.3123788361773547) circle (1.5pt);
\draw[color=red] (4.47,2.5) node {$B_1$};
\draw [fill=red] (-1.6244874871408412,3.7332713069613477) circle (1.5pt);
\draw[color=red] (-1.67,4.03) node {$B_3$};
\draw [fill=red] (-5.236063578787058,-0.005090279250499996) circle (1.5pt);
\draw[color=red] (-5.55,0) node {$B_0$};
\draw [fill=red] (-1.6115777718049742,-3.7364172725503857) circle (1.5pt);
\draw[color=red] (-1.71,-4) node {$B_2$};
\draw [fill=red] (4.240053740297761,-2.304142591337816) circle (1.5pt);
\draw[color=red] (4.5,-2.42) node {$B_4$};
\end{scriptsize}
\end{tikzpicture}
    \caption{Two pentagons are inscribed in $\mathcal D_1$ and $\mathcal D_2$, and circumscribed about $\mathcal D$ homothetic to both $\mathcal D_1$ and $\mathcal D_2$.}
    \label{fig:homellips}
\end{figure}
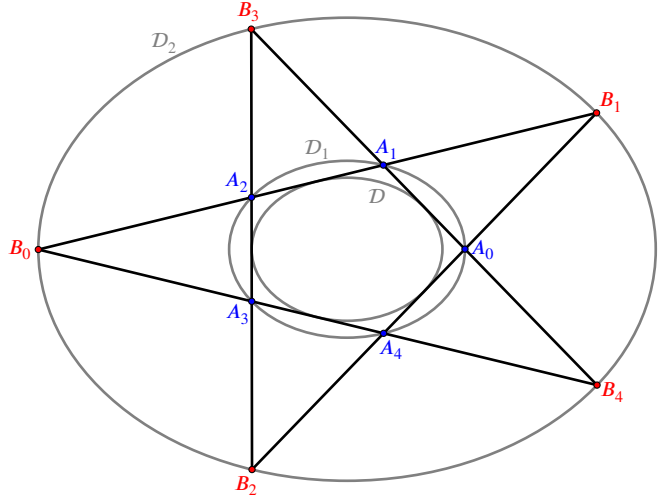

\begin{corollary}\label{cor:pentagonpentagram}
Let $\mathcal D_1$, $\mathcal D_2$, and $\mathcal D$ be three homothetic ellipses. Suppose that $P_1$ is a pentagon of density $1$ inscribed in $\mathcal D_1$ and circumscribed about $\mathcal D$, and that $P_2$ is a pentagon of density $2$ inscribed in $\mathcal D_2$ and circumscribed about $\mathcal D$. Then the total areas of the power circles associated with $P_1$ and $P_2$ are equal. (See Figure \ref{fig:homellips}.)
\end{corollary}

\begin{proof}
Follows from Corollary \ref{cor:mediansum} and Example \ref{ex:equalms}.
\end{proof}

The affine averaging principle developed here shows that rotational symmetry and affine \,\,geometry together provide a simple mechanism for generating quadratic invariants of odd Poncelet polygons. It would be interesting to determine whether analogous affine averaging principles exist for other geometric quantities arising in Poncelet families.

\vspace{0.5cm}
\noindent
\thanks{\textbf{Acknowledgments}. The author is grateful to his wife Saba Fatema for her continuous support, encouragement and helpful discussions.

\bibliographystyle{amsplain}  
\bibliography{references}
\end{document}